\documentclass{amsart}
\usepackage{amsmath,amssymb,amscd,latexsym}

\usepackage[all]{xy}
\newcommand{\A}{{\mathbb{A}}}
\newcommand{\C}{{\mathbb{C}}}

\newcommand{\colim}{\operatorname*{colim}}
\newcommand{\hocolim}{\operatorname*{hocolim}}
\newcommand{\holim}{\operatorname*{holim}}
\newcommand{\Hom}{\mathrm{Hom}}
\newcommand{\Imm}{\mathrm{Im}\,}
\newcommand{\Sm}{\mathbf{Sm}}

\newcommand{\Smb}{\overline{\Sm}}

\newcommand{\Smknis}{\Sm_{k,\mathrm{Nis}}}

\newcommand{\sPre}{\mathbf{sPre}}

\newcommand{\sPresm}{\sPre(\Sm)}
\newcommand{\sPresmb}{\sPre(\Smb)}
\newcommand{\sPresmc}{\sPre(\Sm_{\C})}
\newcommand{\sPreBh}{\sPre(\Bh)}

\newcommand{\sPreTh}{\sPre(\Th)}
\newcommand{\sPreVh}{\sPre(\Vh)}
\newcommand{\hosPresm}{\mathrm{ho}\sPresm}
\newcommand{\hosPresmb}{\mathrm{ho}\sPresmb}
\newcommand{\hosPresmc}{\mathrm{ho}\sPresmc}
\newcommand{\hosPreBh}{\mathrm{ho}\sPreBh}

\newcommand{\hosPreTh}{\mathrm{ho}\sPreTh}
\newcommand{\hosPreVh}{\mathrm{ho}\sPreVh}
\newcommand{\sS}{\mathbf{sS}}

\newcommand{\Ah}{{\mathcal A}}
\newcommand{\oAh}{\overline{\Ah}}
\newcommand{\Bh}{\mathcal{B}}
\newcommand{\Ch}{{\mathcal C}}

\newcommand{\Fh}{{\mathcal F}}

\newcommand{\Gh}{{\mathcal G}}

\newcommand{\Kh}{\mathcal{K}}

\newcommand{\Th}{{\mathcal T}}
\newcommand{\Uh}{{\mathcal U}}

\newcommand{\Vh}{{\mathcal V}}

\newcommand{\oU}{\overline{U}}

\newcommand{\oX}{\overline{X}}
\newcommand{\oY}{\overline{Y}}

\newcommand{\oOmega}{\overline{\Omega}}

\newcommand{\into}{\hookrightarrow}

\newtheorem{theorem}{Theorem}[section]
\newtheorem{lemma}[theorem]{Lemma}
\newtheorem{prop}[theorem]{Proposition}
\newtheorem{cor}[theorem]{Corollary}
\theoremstyle{definition}
\newtheorem{defn}[theorem]{Definition}

\newtheorem{remark}[theorem]{Remark}

\begin{document}
\title{Homotopy theory of smooth compactifications of algebraic varieties}
%
\author{Gereon Quick}
\address{Mathematisches Institut, WWU M\"unster, Einsteinstr. 62, 48149 M\"unster, Germany}
\email{gquick@math.uni-muenster.de}
\date{}
\begin{abstract}
In this short note we show that the homotopy category of smooth compactifications of smooth algebraic varieties is equivalent to the homotopy category of smooth varieties over a field of characteristic zero.  As an application we show that the functor sending a variety to the p-th step of the Hodge filtration of its complex cohomology is representable in the homotopy category of simplicial presheaves on smooth complex varieties.  The main motivation are recent applications of homotopy theory to Deligne-Beilinson cohomology theories.
\end{abstract}

\maketitle

\section{Introduction}

Let $X$ be a smooth algebraic variety. In order to have well-behaved algebraic invariants it is often necessary to assume that $X$ behaves like a compact manifold, i.e. that $X$ is complete or projective. If $X$ is not complete, several kinds of  unpleasant phenomena may occur. For example the cohomology groups may be infinite-dimensional or, over the complex numbers, the Hodge filtration on the de Rham cohomology may contain very little information.

There are different techniques to remedy these defects. Over a field $k$ of characteristic zero, Hironaka's theorem on resolutions of singularities shows that every smooth variety over $k$ has a smooth compactification. This means that there is a smooth projective variety over $k$ which contains $X$ as an open subset such that the complement $D=\oX-X$ is a normal crossing divisor. If $k=\C$ such a smooth compactification provides a way to define the Hodge filtration on the complex cohomology of $X$ (see \cite{hodge2}). Beilinson used smooth compactifications to define a new version of Deligne cohomology (see \cite{beilinson}). In \cite{beilinson2}, Beilinson applies smooth compactifications and several arithmetic variations of this notion to study derived de Rham cohomology.

In recent approaches homotopy theory of simplicial presheaves is used to provide applications and generalizations of Deligne-Beilinson cohomology (see \cite{bunketamme} for applications to the regulator map in algebraic $K$-theory, \cite{holmscholbach} for a construction of motivic Arakelov cohomology, and \cite{hfcbordism} for a generalization of Deligne-Beilinson cohomology with potential applications to the cycle map). All these approaches require that the functor $X\mapsto F^pH^*(X;\C)$ which sends a complex variety $X$ to the $p$-th step of the Hodge filtration of its complex cohomology is representable in the homotopy category. By working with presheaves on the big site of smooth varieties, one would like to avoid the choice of a particular compactification. The obvious solution is to take into account all compactifications at once. 

The purpose of this short note is to show that the homotopy category of smooth compactification and the homotopy categories of smooth varieties are in fact equivalent. 
In more detail, let $\Sm=\Sm_k$ be the category of smooth varieties over a field $k$ of characteristic zero and $\Smb$ be the category of pairs $(X,\oX)$ of a smooth variety $X$ together with a smooth compactification $\oX$.  We show that the forgetful functor $u:\Smb \to \Sm$, $(X,\oX) \mapsto X$, induces a pair of  Quillen equivalences 
\[
Lu^{\ast}:\hosPresmb \stackrel{\sim}{\leftrightarrow} \hosPresm:Ru_{\ast} 
\]
on the homotopy categories of simplicial presheaves. This is a reinterpretation of the insight of Beilinson in \cite{beilinson2} in terms of the homotopy theory of simplicial presheaves. 

As an application of this observation we study in the last section how, for every $p\ge 0$, the functor $X\mapsto F^pH^*(X;\C)$ is representable in $\hosPresmc$.

%
%

\section{Generalized bases of topologies and homotopy theory}

\subsection{Simplicial presheaves and hypercovers}

Let $\Th$ be an essentially small site. We denote by $\sPreTh$ the category of simplicial presheaves on $\Th$. There are several important model structures on the category $\sPreTh$ (see \cite{jardine}, \cite{blander}, \cite{dugger}, \cite{isaksen}.).

We will consider the local projective model structure as in \cite{blander} and \cite{dugger}. We start with the projective model structure on $\sPreTh$. A map $\Fh \to \Gh$ in $\sPreTh$ is called a 
\begin{itemize}
\item projective weak equivalence if $\Fh(X) \to \Gh(X)$ is a weak equivalence in $\sS$ for every $X\in \Th$;
\item projective fibration if every $\Fh(X) \to \Gh(X)$ is a fibration in $\sS$;
\item projective cofibration if it has the left lifting property with respect to all trivial fibrations.
\end{itemize}

In order to obtain a local model structure, i.e. one which respects the topology on the site $\Th$, we can localize the projective model structure at a specific type of hypercovers. We briefly recall the most important notions.

Recall from \cite{verdier1} that a map $f: \Fh \to \Gh$ of presheaves on $\Th$ is called a {\it generalized cover} if for any map $X\to \Fh$ from a representable presheaf $X$ to $\Fh$ there is a covering sieve $R \into X$ such that for every element $U \to X$ in $R$ the composite $U\to X \to \Fh$ lifts through $f$. 

Moreover, Dugger and Isaksen \cite[\S 7]{di} give the following characterization of local acyclic fibrations and hypercovers. A map $f: \Fh \to \Gh$ of simplicial presheaves on $\Th$ is a {\it local acyclic fibration} if for every $X\in \Th$ and every commutative diagram 
\[
\xymatrix{
\partial\Delta^n \otimes X \ar[r] \ar[d] & \Fh \ar[d]\\
\Delta^n\otimes X \ar[r] & \Gh}
\]
there exists a covering sieve $R\into X$ such that for every $U\to X$ in $R$, the diagram one obtains from restricting from $X$ to $U$ 
\[
\xymatrix{
\partial\Delta^n \otimes U \ar[r] \ar[d] & \Fh \ar[d]\\
\Delta^n\otimes U \ar@{.>}[ur] \ar[r] & \Gh}
\]
has a lifting $\Delta^n\otimes U \to \Fh$. Note that this implies in particular that the map $\Fh_0 \to \Gh_0$ of presheaves is a generalized cover.

\begin{defn}
Let $X$ be an object of $\Th$ and let $\Uh$ be a simplicial presheaf on $\Th$ with an augmentation map $\Uh \to X$ in $\sPreTh$. This map is called a {\it hypercover} of $X$ if it is a local acyclic fibration and each $\Uh_n$ is a coproduct of representables. 
\end{defn}

If $\Uh \to X$ is a hypercover, then the map $\Uh_0 \to X$ is a cover in the topology on $\Th$. Moreover, the map $\Uh_1 \to \Uh_0 \times_X \Uh_0$ is a generalized cover. In general, for each $n$, the face maps combine such that $\Uh_n$ is a generalized cover of a finite fiber product of different $\Uh_k$ with $k < n$.

Since the projective model structure on $\sPreTh$ is cellular, proper and simplicial, it admits a left Bousfield localization with respect to all maps 
\[
\{\hocolim \Uh_* \to X\}
\] 
where $X$ runs through all objects in $\Th$ and $\Uh$ runs through the hypercovers of $X$. The resulting model structure is the {\it local projective model structure} on $\sPreTh$ (see \cite{blander} and \cite{dugger}). The weak equivalences, fibrations and cofibrations in the local projective model structure are called {\it local weak equivalences}, {\it local projective fibrations} and {\it local projective cofibrations} respectively. The local acyclic fibrations that were defined above are exactly the maps which are both local weak equivalences and local projective fibrations (see \cite[\S 3]{dhi}). We denote the resulting homotopy category by $\hosPreTh$.

Dugger, Hollander and Isaksen showed that the fibrations in the local projective model structure on $\sPreTh$ have a nice characterization (see \cite[\S\S 3+7]{dhi}). Let $\Uh \to X$ be a hypercover in $\sPreTh$ and let $\Fh$ be a projective fibrant simplicial presheaf. Since each $\Uh_n$ is a coproduct of representables, we can form a product of simplicial sets $\prod_a \Fh((U_n)^a)$ where $a$ ranges over the representable summands of $\Uh_n$. The simplicial structure of $\Uh$ defines a cosimplicial diagram in $\sS$  
\[
\prod_a \Fh(U_0^a) \rightrightarrows \prod_a \Fh(U_1^a) \overset{\longrightarrow}{\underset{\longrightarrow}{\to}} \cdots 
\]
The homotopy limit of this diagram is denoted by $\holim_{\Delta}\Fh(\Uh)$.

Following \cite[Definition 4.3]{dhi} we say that a simplicial presheaf $\Fh$ satisfies descent for a hypercover $\Uh \to X$ if there is a projective fibrant replacement $\Fh \to \Fh'$ such that the natural map 
\begin{equation}\label{descentmap}
\Fh'(X) \to \holim_{\Delta}\Fh'(\Uh)
\end{equation}
is a weak equivalence. It is easy to see that if $\Fh$ satisfies descent for a hypercover $\Uh \to X$, then the map \eqref{descentmap} is a weak equivalence for {\it every} projective fibrant replacement $\Fh \to \Fh'$. By \cite[Corollary 7.1]{dhi}, the local projective fibrant objects in $\sPreTh$ are exactly those simplicial presheaves which are projective fibrant and satisfy descent with respect to all hypercovers $\Uh \to X$. For our final applications we will need the following facts. 

\begin{lemma}\label{descent}
Let $\Fh$ be a simplicial presheaf that satisfies descent with respect to all hypercovers. Then every fibrant replacement $\Fh \to \Fh_f$ in the local projective model structure is a projective weak equivalence, i.e. for every object $X\in \Th$ the map 
\[
\Fh(X) \to \Fh_f(X)
\]
is a weak equivalence of simplicial sets.
\end{lemma}
\begin{proof}
Let $g:\Fh \to \Fh'$ be a fibrant replacement in the projective model structure. Since $g$ is a projective acyclic cofibration, it is also a local projective acyclic cofibration. Hence in order to show that $g$ is a fibrant replacement in the local projective model structure, it suffices to show that $\Fh'$ is  fibrant in the local projective model structure. But by the assumptions, $\Fh'$ is projective fibrant and satisfies descent with respect to all hypercovers. By \cite[Corollary 7.1]{dhi}, this show that $\Fh'$ is local projective fibrant. 
Moreover, if $h: \Fh \to \Fh_f$ is any local projective fibrant replacement, then, since $g$ is a local acyclic cofibration, there is an induced map $j:\Fh' \to \Fh_f$ that makes the triangle
\[
\xymatrix{
\Fh \ar[r]^h \ar[d]_g & \Fh_f \\
\Fh' \ar@{.>}[ur]_j & }
\]
commute. By the two-out-of-three property for weak equivalences we know that $j$ is a local weak equivalence as well. But since the local projective model structure is the left Bousfield localization of the projective model structure, a local weak equivalence between local fibrant objects is a projective weak equivalence.  Hence $h$ is a projective weak equivalence too. 
\end{proof}


\begin{prop}\label{cordescent}
Let $\Fh$ be a simplicial presheaf that satisfies descent with respect to all hypercovers and let $X$ be an object of $\Th$. Then, for every projective fibrant replacement $g:\Fh \to \Fh'$, the natural map 
\[
\Hom_{\hosPreTh}(X, \Fh) \to \pi_0(\Fh'(X))\]
is a bijection.  
\end{prop}
\begin{proof}
By Lemma \ref{descent}, the map $g: \Fh \to \Fh'$ is a local projective fibrant replacement in $\sPreTh$. Hence, since the representable presheaf $X$ is local projective cofibrant, we can compute the morphisms from $X$ to $\Fh$ in $\hosPreTh$ via the isomorphism 
\[
\Hom_{\hosPresm}(X, \Fh) \cong \pi(X,\Fh')
\]
where $\pi(X,\Fh')$ denotes the set of simplicial homotopy classes of maps. Since $X$ is representable, this set is just given by the set of connected components of the simplicial set $\Fh'(X)$, i.e. we have $\pi(X,\Fh') \cong \pi_0(\Fh'(X))$. 
\end{proof}


\subsection{Generalized base of a topology and local model structures}\label{basissection}

Let $\Vh$ be an essentially small site. We recall Beilinson's generalization of the notion of a base for $\Vh$ from \cite[\S 2.1]{beilinson2}. 

\begin{defn}\label{basis}
A {\it base for $\Vh$} is a pair $(\Bh, u)$, where $\Bh$ is an essentially small category and $u:\Bh \to \Vh$ is a faithful functor that satisfies the following property: 

(*) For any $V\in \Vh$ and a finite family of pairs $(B_{\alpha}, f_{\alpha})$, $B_{\alpha} \in \Bh$, $f_{\alpha}: V \to u(B_{\alpha})$, there exists a set of objects $B'_{\beta}\in \Bh$ and a covering family $\{u(B'_{\beta}) \to V\}$ such that every composition $u(B'_{\beta}) \to V \to u(B_{\alpha})$ lies in 
\[
\Hom_{\Bh}(B'_{\beta},B_{\alpha})\subset \Hom_{\Vh}(u(B'_{\beta}), u(B_{\alpha})).
\] 
\end{defn}

Our main example for this situation is the forgetful functor $u$ from the category of smooth compactifications $\Bh=\Smb$ to the site $\Vh=\Sm$ of smooth varieties over a field of characteristic zero. We will discuss this example in detail in the next section.

\begin{defn}\label{coveringsieves}
Let $(\Bh, u)$ be a base for $\Vh$. A {\it covering sieve} in $\Bh$ is defined to be a sieve whose image under $u$ is a covering family in $\Vh$. 
\end{defn}

By \cite[Proposition 2.1]{beilinson2}, covering sieves in $\Bh$ form a Grothendieck topology on $\Bh$ and the functor $u$ is continuous. Moreover, $u$ induces an equivalence of toposes. 

Our goal is to transfer these results to a comparison of homotopy categories. Let $\sPreBh$ and $\sPreVh$ denote the categories of simplicial presheaves on $\Bh$ and $\Vh$ respectively. 
The functor $u$ induces a pair of adjoint functors 
\[
u^{\ast}:\sPreBh \leftrightarrow \sPreVh:u_{\ast}. 
\]
The right adjoint $u_*$ is given by composition with $u$, i.e. 
\[
u_*\Fh(B) = \Fh(u(B)).
\]
The left adjoint $u^*$ is defined by sending a simplicial presheaf $\Gh$ on $\Bh$ to the simplicial presheaf on $\Vh$
\[
V \mapsto u^*\Gh(V) := \colim_{(B,f)\in C(V)} \Gh(B)
\]
where the colimit is taken over the category $C(V)$ of pairs $(B,f)$, $B\in \Bh$, $f:V\to u(B)$, with $\Hom_{C(V)}((B,f),(B',f')):= \{g\in \Hom(B',B):u(g)\circ f'=f\}$, and we set $\Gh(B,f):=\Gh(B)$.

Since the functor $u: \Bh \to \Vh$ is continuous, the left adjoint $u^*$ preserves generalized covers of presheaves. Nevertheless, $u^*$ does not in general preserve hypercovers of simplicial presheaves. But if we assume that $u^*$ sends hypercovers in $\sPreBh$ to hypercovers in $\sPreVh$, then we get the following homotopy analogue of \cite[Theorem 4.1]{verdier2} and \cite[Proposition 2.1]{beilinson2}.

\begin{theorem}\label{modelthm}
Let $\Vh$ be an essentially small site and $(\Bh, u)$ be a base for $\Vh$. We assume that $u^*$ preserves finite limits. Then the pair $(u^*,u_*)$ is a Quillen equivalence.  
Hence we obtain adjoint derived equivalences 
\[
Lu^{\ast}:\hosPreBh \stackrel{\sim}{\leftrightarrow} \hosPreVh:Ru_{\ast}. 
\]
\end{theorem}
\begin{proof}
The right adjoint $u_*$ preserves projective weak equivalences and projective fibrations. Hence $(u^*,u_*)$ is a Quillen pair of adjoint functors on the projective model structure. 
Since $u^*$ preserves hypercovers, it is a Quillen left adjoint on the local projective model structure as well. 
It remains to show that $(u^*,u_*)$ is a pair of Quillen equivalences. This follows from Beilinson's argument. Let $a_{\Fh}:\Fh \to u_*u^*\Fh$, for $\Fh \in \sPreBh$, and $b_{\Gh}: u^*u_*\Gh \to \Gh$, for  $\Gh \in \sPreVh$, denote the adjunction maps. 
By \cite[Proposition 2.1]{beilinson2}, the sheafifications $\widetilde{a_{\Fh}}$ and $\widetilde{b_{\Gh}}$ of these maps in the topologies of $\Bh$ and $\Vh$ respectively are isomorphisms. In particular, $\widetilde{a_{\Fh}}$ and $\widetilde{b_{\Gh}}$ are local weak equivalences. Since the map from a simplicial presheaf to its sheafification is a local weak equivalence, the two-out-of-three property for local weak equivalences shows that $a_{\Fh}$ and $b_{\Gh}$ are local weak equivalences. Hence $u^*$ and $u_*$ are Quillen equivalences. 
\end{proof}

\begin{remark}
In the statement of Theorem \ref{modelthm} we required that $u^*$ preserves finite limits. But the proof only uses that $u^*$ preserves hypercovers. Since we do not know of any example of interest where $u^*$ preserves hypercovers but not finite limits, the theorem is formulated with the assumption that $u^*$ preserves finite limits. 
\end{remark}


%
\section{Smooth compactifications}

\subsection{The forgetful functor}

We now apply the arguments from the preceding section to our main example. 
Let $k$ be a field of characteristic zero. We will use the term {\em variety} for a separated scheme of finite type over $k$. Let $\Smknis=\Sm$ be the site of smooth varieties over $k$ with the Nisnevich topology.

Let $\Smb$ be the category whose objects are {\em smooth
compactifications}, i.e. pairs $(X,\oX)=(X\subset \oX)$ consisting of a
smooth variety $X$ embedded as an open subset of a projective variety
$\oX$ and having the property that $\oX-X$ is a normal crossing divisor which is the union of smooth divisors.  A map from $(X, \oX)$ to $(Y,\oY)$ is a commutative diagram 
\[
\xymatrix{
X  \ar[r]\ar[d]  &  \oX \ar[d] \\
Y  \ar[r]        &\oY.}
\]

By Hironaka's theorem \cite{hironaka}, every smooth variety over $k$ admits a smooth compactification. Moreover, for a given smooth variety $X$, the category $C(X)$ of all smooth compactifications of $X$ is filtered. For if $\oX_1$ and $\oX_2$ are two smooth compactifications of $X$, we define $\oX$ to be a resolution of singularities of the closure of the image of the diagonal of $X$ in $\oX_1 \times \oX_2$. Then $\oX$ is a smooth compactification of $X$ with morphisms 
\[
\xymatrix{
\oX \ar[r] \ar[dr] & \oX_1 \\
 & \oX_2}
\]
which induce the identity on $X$. Similarly, if we are given two morphisms 
\[
j_1, j_2: \oX' \rightrightarrows \oX'' 
\]
of smooth compactifications of $X$, then, by taking a resolution of singularities of the closure of the diagonal as before, we can construct a third smooth compactification $\oX$ of $X$ such that there is a commutative diagram in $\Smb$ 
\[
\xymatrix{
\oX \ar[d] \ar[r] & \oX' \ar[d] \\
\oX' \ar[r] & \oX''.}
\]

The forgetful functor 
\begin{align*}
u:\Smb & \to \Sm \\
(X, \oX) &\mapsto X
\end{align*}
induces as above a pair of adjoint functors on the categories of simplicial presheaves  
\[
u^{\ast}:\sPresmb \leftrightarrow \sPresm:u_{\ast}. 
\]
The left adjoint $u^*$ is given by sending a simplicial presheaf $\Fh$ on $\Smb$ to the simplicial presheaf 
\[
X \mapsto u^*\Fh(X) = \colim_{C(X)} \Fh(\oX).
\]

The following proposition can be proved as in \cite[\S 2.5]{beilinson2}. 
\begin{prop}\label{baseprop}
The pair $(\Smb,u)$ is a base for the site $\Sm$ with the Nisnevich topology.
\end{prop}
\begin{proof}
Let $X \in \Sm$ and $\{(U_{\alpha},\oU_{\alpha})\}$ be a finite collection of pairs in $\Smb$ with maps $f_{\alpha}:X \to U_{\alpha}$ in $\Sm$. To show that $(\Smb,u)$ is a base for $\Sm$, it suffices to embed $X$ into a pair $(X,\oX)$ such that the maps $f_{\alpha}$ extend to maps $(X,\oX) \to (U_{\alpha},\oU_{\alpha})$. Therefore, let $\oX'$ be a smooth compactification of $X$. Then let $\oX$ be a smooth compactification of the closure of the image of $X$ in $\oX'\times \prod_{\alpha}\oU_{\alpha}$. The pair $(X,\oX)$ is equipped with induced maps $(X,\oX) \to (U_{\alpha},\oU_{\alpha})$ which extend the $f_{\alpha}$. 
\end{proof}

\begin{lemma}\label{f*hyper}
The left adjoint $u^*: \sPresmb \to \sPresm$ preserves finite limits. In particular, $u^*$ preserves hypercovers. 
\end{lemma}
\begin{proof}
Let $\lim_{i\in I}\Fh_i$ be the limit of a finite diagram in $\sPresmb$. For a variety $X\in \Sm$, we have
\[
u^*(\lim_i\Fh_i)(X) = \colim_{C(X)}((\lim_i\Fh_i)(\oX)).
\]
Since the limit of simplicial presheaves can be calculated objectwise, and since the finite limit commutes with the filtered colimit over $C(X)$, we get
\[
u^*(\lim_i\Fh_i)(X) \cong \lim_i (\colim_{C(X)}\Fh_i(\oX)) = \lim_i u^*\Fh_i(\oX).
\]
Hence obtain an isomorphism of simplicial presheaves 
\[
u^*(\lim_i\Fh_i) \cong \lim_i (u^*\Fh_i).
\]
\end{proof}

As a consequence of Theorem \ref{modelthm} and Lemma \ref{f*hyper} we get the following result.

\begin{cor}
The pair 
\[
u^{\ast}:\sPresmb \leftrightarrow \sPresm:u_{\ast} 
\]
is a Quillen equivalence. Hence we obtain adjoint derived equivalences 
\[
Lu^{\ast}:\hosPresmb \stackrel{\sim}{\leftrightarrow} \hosPresm:Ru_{\ast}. 
\]
\end{cor}

\begin{remark}
The results of this section do not depend on the fact that we are working in the Nisnevich topology. In \cite{beilinson2}, Beilinson uses the $h$-topology generated by covering families which are universal topological epimorphisms. Our choice of the Nisnevich topology here is due to the situation of the next section which is related to the homotopy theoretical approach to generalized Deligne-Beilinson cohomology theories and its connections to motivic homotopy theory in \cite{hfcbordism}.
\end{remark}


\subsection{Logarithmic differential forms}\label{Omegalog}

Let $k$ be the complex numbers $\C$. 
For a smooth complex variety $X$, let $\oX$ be a smooth compactification of $X$ and let $D:=\oX - X$ denote the complement of $X$. The divisor $D$ is a normal crossing divisor which is the union of smooth divisors. Let $\Omega^1_{\oX}\langle D \rangle$ be the locally free sub-module of $j_*\Omega^1_{X}$ generated by $\Omega^1_{X}$ and by $\frac{dz_i}{z_i}$ where $z_i$ is a local equation for an irreducible local component of $D$. The sheaf $\Omega^p_{\oX}\langle D \rangle$ of meromorphic $p$-forms on $\oX$ with at most logarithmic poles along $D$ is defined to be the locally free sub-sheaf $\bigwedge^p\Omega^1_{\oX}\langle D \rangle$ of $j_*\Omega_{\oX}^p$. 
The Hodge filtration on the complex cohomology of $X$ can be defined as the image 
\begin{equation}\label{Hodgefiltration} 
F^pH^n(X;\C):= \Imm (H^n(\oX; \Omega^{* \geq p}_{\oX}\langle D \rangle) \to H^n(X;\C)).
\end{equation}

It is a theorem of Deligne's \cite[Th\'eor\`eme 3.2.5 and Corollaire 3.2.13]{hodge2} that for smooth complex algebraic varieties, the homomorphism 
\[
H^n(\oX; \Omega^{* \geq p}_{\oX}\langle D \rangle) \to H^n(\oX; \Omega^{*}_{\oX}\langle D \rangle)
\]
is injective and the image is independent of the choice of $\oX$. In particular, the map in \eqref{Hodgefiltration} induces an isomorphism
\begin{equation}\label{delignehodge}
H^n(\oX; \Omega^{* \geq p}_{\oX}\langle D \rangle) \cong F^pH^n(X;\C) \subset H^n(X;\C).
\end{equation}

%

We denote by $\oOmega^*$ the presheaf of differential graded $\C$-algebras on $\Smb$ that sends a pair $X \subset \oX$ with $D:=\oX - X$ to $\Omega_{\oX}^* \langle D \rangle (\oX)$. For any given integer $p\geq 0$, we denote by $\oOmega^{*\geq p}$ the presheaf on $\Smb$ that sends a pair $X\subset \oX$ to $\Omega_{\oX}^{*\geq p} \langle D \rangle (\oX)$.

Let 
\[
\Omega^{\ast\ge p}_{\oX}\langle D \rangle \to \Ah^{*\geq p}_{\oX}\langle D \rangle
\]
be any resolution by cohomologically trivial sheaves which is
functorial in $X\subset \oX$ and which induces a commutative diagram 
\[
\xymatrix{
\Omega^{\ast\ge p}_{\oX}\langle D \rangle (\oX) \ar[r] \ar[d] & \Omega^*_X(X) \ar[d] \\
\Ah^{*\geq p}_{\oX}\langle D \rangle (\oX) \ar[r] & \Ah^*_X(X).}
\]
For example, $\Ah^{*\geq p}_{\oX}\langle D \rangle$ could be the Godemont resolution (\cite[\S 3.2.3]{hodge2}). The reader should note that $\Ah^{*\geq p}_{\oX}\langle D \rangle$ are double complexes, though we will only consider their total complexes. 

We denote the presheaf of complexes on $\Smb$ that sends a pair $(X,\oX)$ to $\Ah^{*\geq p}_{\oX}\langle D \rangle (\oX)$ by $\oAh^{*\geq p}$, and let 
\[
\oOmega^{\ast\ge p}  \to \oAh^{*\geq p} 
\]
be the associated map of complexes of presheaves on $\Smb$.    

Our final goal is to show that the functor $X\mapsto F^pH^n(X;\C)$ is representable in $\hosPresm$. Therefore, we recall the Dold-Kan correspondence on the level of presheaves which connects presheaves of complexes to simplicial presheaves.

For a  chain complex of presheaves of abelian groups $\Ch_*$ on $\Sm$, we denote by $H_i(\Ch^*)$ the presheaf $U\mapsto H_i(\Ch_*(U))$. For a cochain complex $\Ch^*$ we will denote by $\Ch_*$ its associated chain complex given by $\Ch_n:=\Ch^{-n}$. For any given $n$, we denote by $\Ch^*[n]$ the cochain complex given in degree $q$ by $\Ch^q[n]:=\Ch^{q-n}$. 

Applying the normalized chain complex functor objectwise, we obtain a functor $\Gh \mapsto N(\Gh)$ from simplicial sheaves of abelian groups to chain complexes of sheaves of abelian groups. Then we have $\pi_i(\Gh)\cong H_i(N(\Gh))$. The functor has a right adjoint $\Kh$ again obtained by applying the corresponding functor for chain complexes objectwise. 
For a presheaf of chain complexes $\Ch^*$ on $\Sm$ and an integer $n$, we denote by $\Kh(\Ch^*,n) := \Kh(\Ch^*[n])$ the Eilenberg-MacLane simplicial presheaf corresponding to $\Ch^*[n]$ under the Dold-Kan correspondence. 

The following result is an important ingredient in the homotopy theoretical approach to Deligne-Beilinson cohomology in \cite{hfcbordism}. 

\begin{theorem}
\label{compindpairs}
For every smooth complex variety $X$, there is a natural isomorphism 
\[
\Hom_{\hosPresm}(X, u^*\Kh(\oAh^{*\geq p}, n)) \cong F^pH^n(X; \C).
\]
\end{theorem}
\begin{proof}
Let $H^n(\oAh^{*\geq p}(\oX))$ be the $n$th cohomology group of the complex of global sections $\oAh^{*\geq p}(\oX)$. Since the sheaf $\Ah_{\oX}\langle D \rangle ^{*\geq p}$ is cohomologically trivial, 
it follows that $H^n(\oAh^{*\geq p}(\oX))$ is isomorphic to the sheaf cohomology $H^n(X;\oAh^{*\geq p}_{\oX}\langle D \rangle)$ which, by Deligne's isomorphism \eqref{delignehodge} (\cite[Corollaire 3.2.13]{hodge2}) and the choice of $\oAh^{*\geq p}_{\oX}\langle D \rangle$, is  isomorphic to $F^pH^n(X;\C)$ for any compactification $\oX$ of $X$. 
Since each $\Kh(\oAh^{*\geq p},n)(\oX)$ is a Kan complex, the Dold-Kan correspondence implies that, for every $n,q\ge 0$, there is a natural isomorphism  
\[
\pi_q(\Kh(\oAh^{*\geq p},n)(\oX)) \cong H^{n-q}(\oAh^{*\geq p}(\oX)).
\]
Since the category $C(X)$, over which we take the colimit to define $u^*$, is filtered, and since homotopy groups and cohomology commute with filtered colimits, 
this implies that there is a natural isomorphism
\begin{equation}\label{pifpiso}
\pi_q(u^*\Kh(\oAh^{*\geq p},n))(X)) \cong F^pH^{n-q}(X;\C).
\end{equation}
The complex cohomology functor $X\mapsto H^*(X;\C)$ satisfies descent for the Nisnevich topology in the sense that every distinguished square in the Nisnevich topology 

\begin{equation}\label{dsquare}
\xymatrix{
U\times_XV \ar[d] \ar[r] & V \ar[d] \\
U \ar[r] & X}
\end{equation}
induces a long exact Mayer-Vietoris sequence. By \cite[Th\'eor\`eme 1.2.10 and Corollaire 3.2.13]{hodge2}, this sequence respects the Hodge filtration and yields a long exact sequence 
\[
\ldots \to F^pH^{q-1}(X;\C) \to F^pH^q(U\times_X V;\C) \to F^pH^q(U;\C) \oplus F^pH^q(V;\C) \to F^pH^q(X;\C) \to \ldots
\]
Hence the isomorphism \eqref{pifpiso} implies that the simplicial presheaf $\Fh:=u^*\Kh(\oAh^{*\ge p},n)$ satisfies descent for the Nisnevich topology in the sense that every distinguished square in the Nisnevich topology \eqref{dsquare} induces a long exact sequence
\[
\ldots \to \pi_{q+1}(\Fh(X)) \to \pi_q(\Fh(U\times_XV)) \to \pi_q(\Fh(U))\oplus \pi_q(\Fh(V)) \to \pi_q(\Fh(X)) \to \ldots 
\]
By \cite[Lemma 4.2]{blander}, this implies that any projective fibrant replacement $\Fh'$ of the simplicial presheaf $\Fh=u^*\Kh(\oAh^{*\ge p},n)$ is already a local projective fibrant replacement of $\Fh$. Hence $\Fh$ satisfies descent with respect to all Nisnevich hypercovers. 
By Proposition \ref{cordescent} and \eqref{pifpiso}, we get a natural isomorphism 
\[
\Hom_{\hosPresm}(X, u^*\Kh(\oAh^{*\geq p}, n)) \cong F^pH^n(X; \C).
\]
\end{proof}

\begin{remark}
Another way to formulate the argument of the previous proof would be to remark that the isomorphism \eqref{pifpiso} shows that the simplicial presheaf $\Fh=u^*\Kh(\oAh^{*\ge p},n)$ satisfies the Brown-Gersten property for the Nisnevich topology (see \cite{browngersten} for this notion in the Zariski topology and \cite{mv} for the Nisnevich version). As in \cite[\S 3.1]{mv} one could then deduce from this that the set $\Hom_{\hosPresm}(X, \Fh)$ is in natural bijection with $\pi(X,\Fh)$.  
\end{remark}

%
%
%
\bibliographystyle{amsalpha}

\begin{thebibliography}{999999}
%
%
\bibitem{beilinson} A.\,Beilinson, Higher regulators and values of $L$-functions, J. Soviet Math. 30 (1985), 2036-2070. 
%
\bibitem{beilinson2} A.\, Beilinson, $p$-adic periods and derived de Rham cohomology, J. Amer. Math. Soc. 25 (2012), 715-738.
%
\bibitem{blander} B.\,A.\,Blander, Local projective model structures on simplicial presheaves, $K$-Theory 24 (2001), 283-301.
%
%
%
\bibitem{browngersten} K.\,S.\,Brown, S.\,M.\,Gersten, Algebraic $K$-theory as generalized sheaf cohomology. Algebraic $K$-theory, I: Higher $K$-theories (Proc. Conf., Battelle Memorial Inst., Seattle, Wash., 1972), pp. 266-292. Lecture Notes in Math., Vol. 341, Springer, Berlin, 1973.
%
\bibitem{bunketamme} U.\,Bunke, G.\,Tamme, Regulators and cycle maps in higher-dimensional differential algebraic $K$-theory, preprint, arXiv:1209.6451.
%
%
%
%
\bibitem{hodge2} P.\,Deligne, Th\'eorie de Hodge II, Pub. Math. IHES 40 (1971), 5-57.
%
\bibitem{hodge3} P.\,Deligne, Th\'eorie de Hodge III, Pub. Math. IHES 44 (1974), 5-77.
%
\bibitem{dugger} D.\,Dugger, Universal homotopy theories, Adv. Math. 164 (2001), 144-176.
%
\bibitem{dhi} D.\,Dugger, S.\,Hollander, D.\,C.\,Isaksen, Hypercovers and simplicial presheaves, Math. Proc. Cambridge Philos. Soc. 136 (2004), no. 1, 9-51.
%
\bibitem{di} D.\,Dugger, D.\,C.\,Isaksen, Weak equivalences of simplicial presheaves, in: Homotopy theory: relations with algebraic geometry, group cohomology, and algebraic K-theory, 97-113, 
Contemp. Math., 346, Amer. Math. Soc., Providence, RI, 2004.
%
%
%
%
%
\bibitem{hironaka} H.\,Hironaka, Resolution of singularities of an algebraic variety over a field of characteristic zero I, II. Ann. of Math. (2) 79 (1964), 109-203; ibid. (2) 79 (1964), 205-326.
%
\bibitem{holmscholbach} A.\,Holmstrom, J.\,Scholbach, Arakelov motivic cohomology I, preprint, 	arXiv:1012.2523.
%
\bibitem{hfcbordism} M.\,J.\,Hopkins, G.\,Quick, Hodge filtered complex bordism, preprint, arXiv:1212.2173. 
%
\bibitem{isaksen} D.\,C.\,Isaksen, Flasque model structures for simplicial presheaves, K-Theory 36 (2005), 371-395.
%
\bibitem{jardine} J.\,F.\,Jardine, Simplicial presheaves, J. Pure Appl. Algebra 47 (1987), no. 1, 35-87.
%
%
%
\bibitem{mv} F.\,Morel, V.\,Voevodsky, $\A^1$-homotopy theory of schemes, Publ. IHES 90 (1999), 45-143.
%
%
%
\bibitem{verdier1} J.-L.\,Verdier, Topologies et faisceaux, Th\'eorie des topos et cohomologie \'etale de sch\'emas (SGA 4), Tome 1, Lect. Notes in Math. 269, Springer-Verlag, 1972, pp. 219-264.
%
\bibitem{verdier2} J.-L.\,Verdier, Fonctorialit\' e des cat\'egories de faisceaux, Th\'eorie des topos et cohomologie \'etale de sch\'emas (SGA 4), Tome 1, Lect. Notes in Math. 269, Springer-Verlag, 1972, pp. 265-298.
%
\end{thebibliography}

\end{document}